\documentclass[9pt, shortpaper,twoside,web]{ieeecolor}

\usepackage{generic}
\usepackage{cite}
\usepackage{amsmath,amssymb,amsfonts}
\usepackage{algorithmic}
\usepackage{graphicx}
\usepackage{algorithm,algorithmic}
\usepackage{hyperref}
\hypersetup{hidelinks=true}
\usepackage{textcomp}

\usepackage{tikz}
\usepackage[caption=false]{subfig}

\let\oldlabelindent\labelindent  
\let\labelindent\relax  
\usepackage{enumitem}  
\let\labelindent\oldlabelindent

\newtheorem{definition}{Definition}
\newtheorem{remark}{Remark}
\newtheorem{theorem}{Theorem}
\newtheorem{lemma}{Lemma}

\newtheorem{assumption}{Assumption}

\newtheorem{example}{Example}

\DeclareMathOperator{\sos}{SOS}
\DeclareMathOperator{\vecrz}{vec}

\def\BibTeX{{\rm B\kern-.05em{\sc i\kern-.025em b}\kern-.08em
    T\kern-.1667em\lower.7ex\hbox{E}\kern-.125emX}}
\begin{document}
\title{Feedback Stabilization of Polynomial Systems:\\From Model-based to Data-driven Methods}
\author{Huayuan Huang, M. Kanat Camlibel, \IEEEmembership{Senior Member, IEEE}, Raffaella Carloni, \IEEEmembership{Senior Member, IEEE},\\and Henk J. van Waarde, \IEEEmembership{Member, IEEE}
\thanks{This work was supported by the China Scholarship Council 202106340040. Henk J. van
Waarde acknowledges financial support by the Dutch Research Council (NWO) under the Talent Programme Veni Agreement (VI.Veni.22.335).}
\thanks{Huayuan Huang, M. Kanat Camlibel, Raffaella Carloni and Henk J. van Waarde are with the Bernoulli Institute for Mathematics, Computer Science, and Artificial Intelligence, University of Groningen, 9747 AG Groningen, The Netherlands (email: huayuan.huang@rug.nl; m.k.camlibel@rug.nl; r.carloni@rug.nl; h.j.van.waarde@rug.nl)}}

\maketitle

\begin{abstract}
In this study, we propose new global stabilization approaches for a class of polynomial systems in both model-based and data-driven settings. The existing model-based approach guarantees global asymptotic stability of the closed-loop system only when the Lyapunov function is radially unbounded, which limits its applicability. To overcome this limitation, we develop a new global stabilization approach that allows a broader class of Lyapunov function candidates. Furthermore, we extend this approach to the data-driven setting, considering Lyapunov function candidates with the same functional structure. Using data corrupted by bounded noise, we derive conditions for constructing globally stabilizing controllers for unknown polynomial systems. Beyond handling noise, the proposed data-driven approach can be readily adapted to incorporate further prior knowledge of system parameters to reduce conservatism. In both approaches, sum-of-squares relaxation is used to ensure computational tractability of the involved conditions.
\end{abstract}

\begin{IEEEkeywords}
Data-driven control, nonlinear control, polynomial systems, sum of squares. 
\end{IEEEkeywords}

\section{Introduction}
The control of nonlinear systems is a challenging topic within control theory \cite{willems1972dissipative, CLF, ISSsontag, khalil2002nonlinear, DDCVRFT, Originalmiddle, modelfree, pHAjran, DDConlinedirect}, where key methods include stabilization \cite{CLF, Originalmiddle}, output regulation \cite{outputregulationoverview}, and optimal control \cite{adaptiveoptimal}, among others. Stabilization plays a fundamental role, as it serves as the foundation for many other control tasks \cite{ISSsontag, Originalmiddle, outputregulationoverview}. Existing stabilization methods can be roughly classified into two categories: model-based methods, which work with a model of the system dynamics, and direct data-driven methods, which design controllers for unknown systems directly from data. This study contributes to both categories by developing stabilization approaches in model-based and data-driven settings.

In the model-based setting, a variety of stabilization methods has been developed. Some approaches, such as feedback linearization \cite{isidori1985nonlinear}, are based on Lyapunov functions and are effective for systems with specific structures. A more systematic approach for constructing stabilizing controllers is proposed in \cite{CLF}, which uses control-Lyapunov functions. Another group of approaches, employing storage functions \cite{pHAjran}, offers a physics-based framework for controller design, and has demonstrated effectiveness in practical applications \cite{passivitybasedRobotics}. Nevertheless, these approaches share a common drawback: the stabilizers and associated functions are often constructed separately. To address this problem, a density function-based method is proposed in \cite{DensityConvex}, which ensures that the trajectories of the closed-loop system converge to the origin for \textit{almost all} initial states, however without guaranteeing asymptotic stability in the sense of Lyapunov \cite{DensityDual}. Another method tailored to input-affine polynomial systems is introduced in \cite{Originalmiddle}. Representing the system in a so-called linear-like form, the authors design controllers by imposing that a certain state-dependent matrix is negative definite. Using sum-of-squares (SOS) relaxation \cite{parriloSOS}, this approach enables the joint search for controllers and Lyapunov function candidates via linear matrix inequalities. However, global asymptotic stability of the closed-loop system is guaranteed only when the decision variable associated with the Lyapunov function candidates is a constant matrix. 

Model-based methods are useful, but accurate system models are often unavailable in practice. One possibility to address this is to identify the model from data, though this can be challenging, especially when the data are not informative to unique system identification~\cite{DDCInformativity}. As a result, direct data-driven approaches have been developed to design controllers without relying on a system model. Among these, some methods generate controllers that are updated online \cite{modelfree, DDConlinedirect,iPID}, while others construct controllers using offline data \cite{DDCmiddleTAC, DDCDensity}. Inspired by Willems’ fundamental lemma \cite{DDCJWilliems}, several recent works have contributed to the latter category. In \cite{DDCFormulas}, the authors extend a stabilization approach for linear systems to nonlinear cases by assuming that the dynamics of the nonlinear system can be locally represented by a linearized model. While effective in a local region, this approach does not guarantee global behavior. A global stabilization approach, based on the model-based results in \cite{Originalmiddle}, is proposed in \cite{DDCmiddleTAC}. The authors derive data-based linear matrix inequalities to jointly construct stabilizers and Lyapunov functions for input-affine polynomial systems. However, since the unknown parameters are formulated in a matrix form, this approach cannot incorporate additional prior knowledge of system parameters, which may lead to conservatism in practical applications. Another different strategy based on density functions is used in \cite{DDCDensity, paper1density}. Nevertheless, to ensure global asymptotic stability, the Lyapunov function of the closed-loop system with respect to the obtained controller must be constructed in a separate step \cite{paper1density}. 

Despite much research, there is no general methodology for constructing stabilizers for nonlinear systems. In this respect, the control of polynomial systems has gained attention due to its computational tractability \cite{Originalmiddle, jarvis2003lyapunov}. Moreover, polynomial functions commonly arise in physical systems, particularly in mechanical systems \cite{controltrajectoriesSOS}. 

In this study, we focus on stabilization of input-affine polynomial systems, where the vector field and input matrix consist of linear combinations of given polynomials. We observe that, in certain physical systems such as robots with flexible joints \cite{passivitybasedRobotics}, the control input does not affect all derivatives of the state variables directly. Based on this observation, we incorporate the prior knowledge into our approaches. We derive new global stabilization approaches for polynomial systems in both model-based and data-driven settings. Our model-based method extends the approach in \cite{Originalmiddle}, allowing a broader class of Lyapunov function candidates. In the data-driven setting, we assume that the vector field of the system is an unknown linear combination of given polynomials. We formulate conditions that enable synthesis of stabilizing controllers directly from data, based on the new model-based results derived in this paper.

To summarize, our main contributions are the following:
\begin{enumerate}
    \item We derive new feedback stabilization methods for polynomial systems in both model-based and data-driven settings.
    \item In comparison to \cite{Originalmiddle}, our model-based approach guarantees the global asymptotic stability of the closed-loop system without requiring the Lyapunov function to be radially unbounded.
    \item Using data corrupted by bounded noise, we derive design conditions for globally stabilizing controllers for unknown polynomial systems. Our approach incorporates prior knowledge of system parameters, which is demonstrated to reduce conservatism in numerical examples.
\end{enumerate}

The remainder of this paper is organized as follows. In Section \ref{section A model-based stabilization approach}, we present a novel stabilization approach for given polynomial systems. We then extend this approach to the data-driven setting in Section \ref{section Data-driven stabilization}. In Section \ref{section simulation}, our data-driven approach is illustrated with three examples. Finally, we conclude this article in Section \ref{section conclusion}.

\textit{Notation:} For a set of indices $\alpha \subseteq \{1, 2, \dots, p\}$, we define $v_{\alpha}$ as the subvector of $v \in \mathbb{R}^p$ corresponding to the entries indexed by $\alpha$. For a matrix $Y = \begin{bmatrix}
	y_1 & y_2 & \cdots & y_q
\end{bmatrix}$, where $y_i \in \mathbb{R}^p$ for $i = 1,2, \dots, q $, the vectorization of $Y$ is defined as $\vecrz(Y) = \begin{bmatrix}
	y_1^{\top} & y_2^{\top} & \cdots & y_q^{\top}
\end{bmatrix}^{\top}$. We denote the diagonal matrix with $d_1,\dots,d_n \in \mathbb{R}$ on the diagonal by $\operatorname{diag}(d_1,d_2,\dots,d_n)$. The Kronecker product of two matrices $A\in \mathbb{R}^{n\times m}$ and $B\in \mathbb{R}^{p\times q}$ is denoted by $A \otimes B\in \mathbb{R}^{np\times mq}$. The space of real symmetric $n\times n$ matrices is denoted by $\mathbb{S}^n$. A matrix $A \in \mathbb{S}^n$ is called positive definite if $x^{\top}Ax>0$ for all $x\in \mathbb{R}^n \setminus \{0\}$ and positive semidefinite if $x^{\top}Ax \geq 0$ for all $x\in \mathbb{R}^n$. This is denoted by $A>0$ and $A \geq 0$, respectively. Negative definiteness and semidefiniteness are similarly defined and denoted by $A < 0$ and $A \leq 0$, respectively. 

The set of all multivariate polynomials with coefficients in $\mathbb{R}^{p\times q}$ in the variables $x_1,x_2,\dots,x_n$ is denoted by $\mathbb{R}^{p\times q}[x]$, where $x = \begin{bmatrix}
    x_1&x_2& \cdots &x_n
\end{bmatrix}^{\top}$. When $q=1$ we simply write $\mathbb{R}^{p}[x]$, and when $p=q=1$ we use the notation $\mathbb{R}[x]$. A multivariate polynomial $Q \in \mathbb{R}^{q\times q}[x]$ is called a SOS polynomial if there exists a $\hat{Q} \in \mathbb{R}^{p\times q}[x]$ such that $Q(x) = \hat{Q}^{\top}(x)\hat{Q}(x)$. The set of all $q\times q$ SOS polynomials is denoted by $\sos^q[x]$. For vectors $y \in \mathbb{R}^l$ and $z \in \mathbb{R}^m$, the set of all $q\times q$ SOS polynomials in variables $y$ and $z$ is denoted by $\sos^q[y,z]$. In addition, for a function $f : \mathbb{R}^n \rightarrow \mathbb{R}^m$, we define
$$
\frac{\partial f}{\partial x}(x) = \begin{bmatrix}
    \frac{\partial f}{\partial x_1}(x) & \frac{\partial f}{\partial x_2}(x) & \cdots & \frac{\partial f}{\partial x_n}(x)
\end{bmatrix}.
$$
\section{A novel model-based stabilization approach} \label{section A model-based stabilization approach}
Consider the polynomial input-affine system
\begin{equation} \label{poly sys}
    \begin{bmatrix}
        \dot{x}_1(t)\\
        \dot{x}_2(t)
    \end{bmatrix} = \begin{bmatrix}
        A_{1}\\
        A_{2}
    \end{bmatrix}F(x(t)) + \begin{bmatrix}
        0\\
        B_{2}
    \end{bmatrix}G(x(t))u(t),
\end{equation}
where $x_1(t) \in \mathbb{R}^{n_1}$, $x_2(t) \in \mathbb{R}^{n_2}$, $x(t)=\begin{bmatrix}
    x_1^{\top}(t) & x_2^{\top}(t)
\end{bmatrix}^{\top} \in \mathbb{R}^n$ is the state, and $u(t) \in  \mathbb{R}^m$ is the input. The parameter matrices are given by $A_1 \in \mathbb{R}^{n_1 \times f}$, $A_2 \in \mathbb{R}^{n_2 \times f}$ and $B_2 \in \mathbb{R}^{n_2 \times g}$. The functions $F \in \mathbb{R}^{f}[x]$ and $G \in \mathbb{R}^{g\times m}[x]$ represent polynomial mappings of the state. Note that the system \eqref{poly sys} is structured in the sense that the input $u$ only directly influences $\dot{x}_2$. This structure is motivated by the fact that in many physical systems, not all state derivatives are directly affected by the inputs. For example, in a mass-spring system, the state typically consists of the position and velocity of the mass, and the input force only affects the acceleration (i.e., the derivative of the velocity). The main idea of the paper is to exploit the structure of \eqref{poly sys} to develop dedicated control methods in both model-based and data-driven settings.

In the current section, we focus on model-based control, and our objective is to design a state feedback controller $u = K(x)$, such that the origin of the closed-loop system
\begin{equation} \label{closedloopsystem}
        \begin{bmatrix}
        \dot{x}_1(t)\\
        \dot{x}_2(t)
    \end{bmatrix} = \begin{bmatrix}
        A_{1}\\
        A_{2}
    \end{bmatrix}F(x(t)) + \begin{bmatrix}
        0\\
        B_{2}
    \end{bmatrix}G(x(t))K(x(t))
\end{equation}
is globally asymptotically stable.

\subsection{Previous work} \label{section Previous work}
The existing approach in \cite[Thm.~6]{Originalmiddle} defines the Lyapunov function candidate as
\begin{equation} \label{Vx}
    V(x) = Z^{\top}(x) P^{-1}(x_1) Z(x),
\end{equation}
where the matrix $P \in \mathbb{R}^{p \times p}[x_1]$ satisfies
\begin{equation} \label{Px1 pd}
    P(x_1) > 0\ \forall x_1 \in \mathbb{R}^{n_1},
\end{equation}
and $Z \in \mathbb{R}^p[x]$ is a vector of polynomials in $x$ satisfying $F(x)=H(x)Z(x)$ for some $H \in \mathbb{R}^{f \times p}[x]$. The following assumption on $Z$ is made in \cite{Originalmiddle} and is also adopted in the present paper.
\begin{assumption} \label{assumption Z0x0} 
    The vector $Z(x) = 0$ if and only if $x = 0$.
\end{assumption}

The controller is considered to be of the form 
\begin{equation} \label{Kx}
    K(x) := L(x) P^{-1}(x_1) Z(x),
\end{equation}
where $L \in \mathbb{R}^{m \times p}[x]$. Let $x_{1,j}$ denote the $j$-th entry of $x_1$, and $A_{1,j}$ the $j$-th row of $A_{1}$, where $j \in \{1,2,\dots,n_1\}$. Define
\begin{equation} \label{Mx}
    \begin{aligned}
        M(x) := & - \frac{\partial Z}{\partial x}(x) \begin{bmatrix}
            A_{1} & 0\\
            A_{2} & B_{2}
        \end{bmatrix}\begin{bmatrix}
            H(x)P(x_1) \\
            G(x)L(x)
        \end{bmatrix} \\
        & -\left( \frac{\partial Z}{\partial x}(x) \begin{bmatrix}
            A_{1} & 0\\
            A_{2} & B_{2}
        \end{bmatrix}\begin{bmatrix}
            H(x)P(x_1) \\
            G(x)L(x)
        \end{bmatrix}\right)^{\top}\\
        & + \sum_{i = 1}^{n_1}\frac{\partial P}{\partial x_{1,i}}(x_1) A_{1,i}H(x)Z(x).
    \end{aligned}
\end{equation}
Then, it can be verified that
\begin{equation} \label{theorem model based pVpx}
\begin{aligned}
    &\frac{\partial V}{\partial x}(x)\left(\begin{bmatrix}
        A_{1}\\
        A_{2}
    \end{bmatrix}F(x) + \begin{bmatrix}
        0\\
        B_{2}
    \end{bmatrix}G(x)K(x) \right)\\
    =&  - Z^{\top}(x)P^{-1} (x_1) M(x) P^{-1}(x_1)Z(x).
\end{aligned}  
\end{equation}
Now, the main idea of \cite{Originalmiddle} was to find $P$ and $L$ such that \eqref{Px1 pd} holds and
\begin{equation} \label{Qx pd}
    M(x) > 0 \ \forall x \in \mathbb{R}^{n} \setminus \{0\}.
\end{equation}
Indeed, in \cite{Originalmiddle}, it was shown that \eqref{Px1 pd} and \eqref{Qx pd} imply that the controller $K$ in \eqref{Kx} renders the origin asymptotically stable. Moreover, the origin is globally asymptotically stable provided that $P$ is constant and $Z$ is a vector of monomials. This is due to the fact that, in this case, $V$ is radially unbounded. In the case that $P$ is state-dependent, \cite{Originalmiddle} does not provide any guarantees on global asymptotic stability. A possible way to establish such guarantees is to identify conditions on $P$ under which $V$ is radially unbounded. However, it is not straightforward to impose these conditions via SOS programming. To overcome this problem, we derive a new stabilization approach that guarantees global asymptotic stability \emph{without} requiring $P$ to be constant and $V$ to be radially unbounded. 

\subsection{A new approach}
To present our approach, we first propose a new condition for global asymptotic stability. The main feature of this approach is that the Lyapunov function is not required to be radially unbounded.
\begin{lemma} \label{lemma GAS}
    Let $x = 0$ be an equilibrium point of the system 
    \begin{equation} \label{lemma GAS xfx}
        \dot{x} = f(x),
    \end{equation}
    where $f:\mathbb{R}^n \rightarrow \mathbb{R}^n$ is a continuous function. Suppose that there exist a continuously differentiable function $V :\mathbb{R}^n \rightarrow \mathbb{R}$ and constants $c,r>0$ such that 
    \begin{equation}\label{lemma GAS 1}
        V(0) = 0 \text{ and } V(x)>0 \ \forall x \neq 0,
    \end{equation}
    \begin{equation} \label{lemma GAS 2}
        \frac{\partial V}{\partial x}(x)f(x) < 0\quad \forall x \neq 0,
    \end{equation}
    \begin{equation}\label{lemma GAS 3}
        \frac{\partial V}{\partial x}(x)f(x) < -c\quad \forall x \text{ with } \|x\| \geq r.
    \end{equation}
    Then, $x=0$ is globally asymptotically stable.
\end{lemma}
\begin{proof}
    We first define, for constants $\omega, \rho > 0$, the sets $\mathcal{B}(\omega) := \{x \in \mathbb{R}^n \mid ||x|| \leq \omega\}$ and $\Omega(\omega,\rho) := \{x \in \mathcal{B}(\omega) \mid V(x) \leq \rho\}$. 
    
    By \eqref{lemma GAS 1} and \eqref{lemma GAS 2}, it follows from Theorem 4.1 in \cite{khalil2002nonlinear} that the origin of the system \eqref{lemma GAS xfx} is asymptotically stable. Specifically, there exist constants $\omega, \rho > 0$ such that if $x(0) \in \Omega(\omega,\rho)$, then $x(t) \in \Omega(\omega,\rho)$ for all $t \geq 0$, and $x(t) \rightarrow 0$ as $t \rightarrow \infty$.
    
    To prove global asymptotic stability, it suffices to prove that for any initial state, there exists $T \geq 0$ such that $x(T) \in \Omega(\omega,\rho)$. Our strategy is to prove this by contradiction. Suppose that $x(t) \notin \Omega(\omega,\rho)$ for all $t \geq 0$. Since $V(x) = 0$ if and only if $x=0$, and $V(x)$ is continuous, it follows that there exists $\alpha > 0$ such that $\mathcal{B}(\alpha) \subset \Omega(\omega,\rho)$. It is clear that $x(t) \notin \mathcal{B}(\alpha)$ for all $t \geq 0$. We now distinguish between two cases. First, for $\alpha < r$, the continuous function $\frac{\partial V}{\partial x}(x)f(x)$ has a maximum over the compact set $\{x \in \mathbb{R}^n \mid \alpha \leq ||x|| \leq r \}$. Let $-\gamma$ be this maximum. Note that due to \eqref{lemma GAS 2}, $\gamma > 0$. Then, we have
    $$
    \frac{\partial V}{\partial x}(x(t))f(x(t)) \leq -\min\{c,\gamma\} < 0\ \forall  t \geq 0.
    $$
    It follows that 
    $$
    \begin{aligned}
        V(x(t)) &= V(x(0)) + \int_{0}^{t} \frac{\partial V}{\partial x}(x(\tau))f(x(\tau)) d\tau \\
        &\leq V(x(0)) - \min\{c,\gamma\} t.
    \end{aligned}
    $$
    For sufficiently large $t$, the right-hand side is negative. This contradicts \eqref{lemma GAS 1}. Subsequently, consider the case that $\alpha \geq r$. In this case, it follows by \eqref{lemma GAS 3} that
    $$
    \frac{\partial V}{\partial x}(x(t))f(x(t)) \leq -c \ \forall  t \geq 0.
    $$
    The contradiction is obtained in the same manner. Therefore, we conclude that, for any initial state $x(0)\in \mathbb{R}^n$, there exists $T \geq 0$ such that $x(T) \in \Omega(\omega,\rho)$. This proves the lemma.
\end{proof}

Compared to Theorem 4.2 in \cite{khalil2002nonlinear}, Lemma \ref{lemma GAS} actually provides an alternative condition to radial unboundedness of the Lyapunov function. To illustrate this, consider the system \eqref{lemma GAS xfx} with
$$
f(x) = \begin{bmatrix}
    (-x_1 - x_2)(1+x_1^2)^2\\
    x_1-x_2
\end{bmatrix},
$$
and the function 
$$
    V(x) = \frac{x_1^2}{1+x_1^2} + x_2^2.
$$
Clearly, $\frac{\partial V}{\partial x}(x)f(x) = -2x_1^2-2x_2^2$. Therefore, the assumptions of Lemma \ref{lemma GAS} are satisfied, and we conclude that the origin is globally asymptotically stable. However, note that the function $V$ is not radially unbounded. 

With Lemma \ref{lemma GAS}, we present the main result of this section, which can be viewed as an extended version of \cite[Thm.~6]{Originalmiddle}. In contrast to \cite{Originalmiddle}, we do not require a radially unbounded Lyapunov function. Furthermore, the entries of $Z$ are not necessarily monomials.
\begin{theorem} \label{theorem model based}
    For the system \eqref{poly sys}, suppose that the constants $\epsilon_1,c,r > 0$, polynomials $\epsilon_2,\epsilon_3\in\mathbb{R}[x]$, and polynomial matrices $P \in \mathbb{R}^{p \times p}[x_1]$ and $L\in\mathbb{R}^{m \times p}[x]$ satisfy 
    \begin{enumerate} [label=(\alph*)]
        \item \label{theorem model based 1} $\epsilon_2(x) > 0$ for all $x \neq 0$ and $\epsilon_3(x) > 0$ for all $x$,
        \item \label{theorem model based 2} $\epsilon_2(x)\epsilon_3^{-1}(x)Z^{\top}(x)Z(x) \geq c$ for all $x$ with $\|x\| \geq r$,
        \item \label{theorem model based 3} $P(x_1) - \epsilon_1I \in \sos^p[x_1]$,
        \item \label{theorem model based 4} $\begin{bmatrix}
            M(x) & \epsilon_2(x)P(x_1)\\
            \epsilon_2(x)P(x_1) & \epsilon_2(x)\epsilon_3(x)I
        \end{bmatrix} \in \sos^{2p}[x]$,  
    \end{enumerate}
    where $M(x)$ is defined as in \eqref{Mx}. Then, the controller \eqref{Kx} renders the origin globally asymptotically stable.
\end{theorem}
\begin{proof}
    To prove this theorem, we apply Lemma \ref{lemma GAS} with
    \begin{equation} \label{fx}
        f(x) = \begin{bmatrix}
            A_1\\
            A_2
        \end{bmatrix} F(x) + \begin{bmatrix}
            0\\B_2
        \end{bmatrix} G(x)K(x),
    \end{equation}
    and the function $V$ given in \eqref{Vx}. It follows from condition \ref{theorem model based 3} and Assumption \ref{assumption Z0x0} that \eqref{lemma GAS 1} holds. We apply a Schur complement argument to the matrix in condition \ref{theorem model based 4}, which, combined with condition \ref{theorem model based 1}, yields $M(x) \geq \epsilon_2(x)\epsilon_3^{-1}(x)P(x_1)P(x_1)$ for all $x$. Together with \eqref{theorem model based pVpx}, we obtain
    \begin{equation} \label{bridge to ddc}
        \frac{\partial V}{\partial x}(x)f(x) \leq -\epsilon_2(x)\epsilon_3^{-1}(x)Z^{\top}(x)Z(x)\  \forall x,
    \end{equation}
    which implies \eqref{lemma GAS 2}. Finally, by condition \ref{theorem model based 2}, \eqref{lemma GAS 3} holds. This proves the theorem.
\end{proof}

In the following, we illustrate Theorem \ref{theorem model based} with a numerical example (see~\cite{paper1density}).

\begin{example} \label{ex 1}
    In this example, we consider the system
    \begin{equation} \label{ex sys 1}
    \begin{aligned}
        \dot{x}_1 = x_2 - x_1^2,\quad  \dot{x}_2 = u.
    \end{aligned}
    \end{equation} 
    Let 
    $$
    F(x) = \begin{bmatrix}
        x_1^2 & x_2
    \end{bmatrix}^{\top}\text{ and } G(x) = 1.
    $$
    Then, 
    $$
    A_1 = \begin{bmatrix}
        -1 & 1
    \end{bmatrix},\ A_2 = \begin{bmatrix}
        0 & 0
    \end{bmatrix} \text{ and }B_2 = 1.
    $$
    In case $Z(x) = x$, it was shown in \cite{paper1density} that for any choice of $H(x)$, the matrix inequality \eqref{Qx pd} does not hold with a constant $P$. This implies that the approach in \cite{Originalmiddle} cannot yield a globally stabilizing controller.

    Motivated by this observation, we next apply Theorem \ref{theorem model based}. We choose $Z(x) = x$ and then
    $$
    H(x) = \begin{bmatrix}
        x_1 & 0\\
        0 & 1
    \end{bmatrix}.
    $$
    Furthermore, we select $\epsilon_1 = 0.1$, $\epsilon_2 = 0.01$ and $\epsilon_3(x) = 1 + x_1^2$, from which it follows that conditions \ref{theorem model based 1} and \ref{theorem model based 2} of Theorem \ref{theorem model based} hold. Moreover, it can be verified that conditions \ref{theorem model based 3} and \ref{theorem model based 4} of Theorem \ref{theorem model based} hold if 
    $$
    \begin{aligned}
        P(x_1) & = \begin{bmatrix}
            1 & x_1-0.5\\
            x_1-0.5 & 2(x_1-0.5)^2+3
        \end{bmatrix},\\
        L(x) & = \begin{bmatrix}
            -2+2x_1+x_2-2x_1^2 & -2-x_2+2x_1x_2-2x_1^3
        \end{bmatrix}.
    \end{aligned}
    $$
    With $P(x_1)$ and $L(x)$ in place, we calculate
    \begin{equation} \label{ex1 Vx}
        V(x) = Z^{\top}(x)P^{-1}(x_1)Z(x) = \frac{\eta(x)}{\det{P(x_1)}},
    \end{equation}
    \begin{equation} \label{ex1 Kx}
        K(x) = L(x)P^{-1}(x_1)Z(x) = \frac{\xi(x)}{\det{P(x_1)}},
    \end{equation}
    where
    $$
    \begin{aligned}
        \eta(x) = &\ 3.5x_1^2 + x_1x_2 + x_2^2 - 2x_1^3 - 2x_1^2x_2 + 2x_1^4 ,\\
        \xi(x) = & -8x_1-3x_2+13x_1^2+6x_1x_2-0.5x_2^2-15x_1^3 \\
         & -3x_1^2x_2+x_1x_2^2+7x_1^4-2x_1^5,
    \end{aligned}
    $$
    and
    $$
    \det{P(x_1)} = 3.25 - x_1 + x_1^2
    $$
    denotes the determinant of $P(x_1)$. This verifies global asymptotic stability of the closed-loop system with respect to the system \eqref{ex sys 1} and the controller $K(x)$. The phase portrait of the closed-loop system and the level sets of $V(x)$ are illustrated in Fig.~\ref{fig 1}.
    \begin{figure}[!t]
	\centering
	\includegraphics[clip,trim=40 30 40 60,width=0.85\columnwidth]{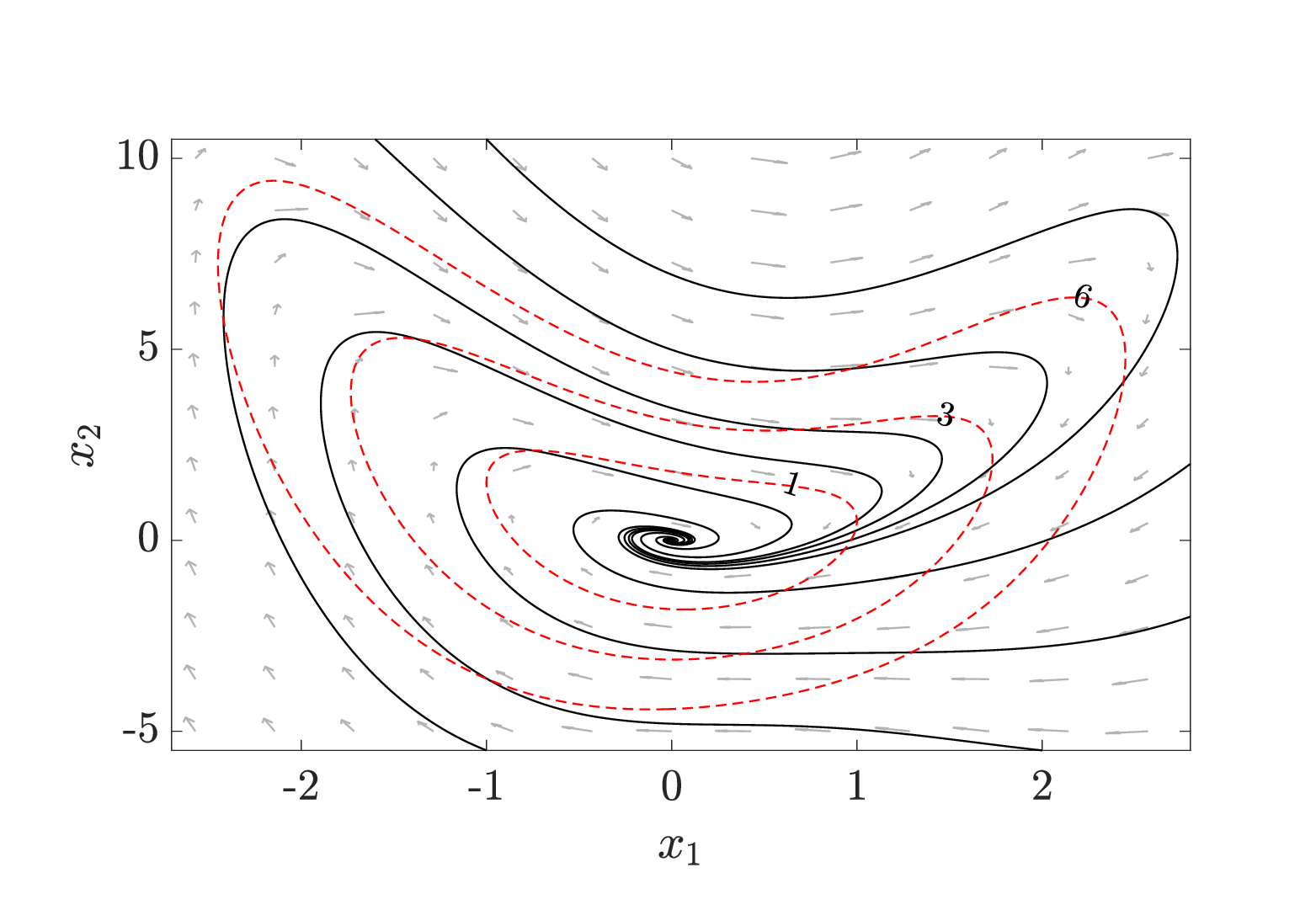}
        \caption{Phase portrait of the closed-loop system and the level sets of $V(x)$ in Example~\ref{ex 1}. The gray arrows depict the closed-loop vector field and the black lines are trajectories starting from the edges and converging to the origin. The level sets are indicated by red dashed lines.}
	\label{fig 1}
    \end{figure}
\end{example}

\section{Data-driven stabilization} \label{section Data-driven stabilization}
In the previous section, we proposed a global stabilization approach for a class of polynomial systems. Next, we extend this approach to the data-driven setting. Specifically, we aim to find a globally stabilizing controller for polynomial systems using data.

\subsection{Problem formulation} \label{section Problem formulation}
Given $F$ and $G$, we denote the system of the form \eqref{poly sys} by $(A_1,A_2,B_2)$. In this section, we consider a true system $(A_{s1},A_{s2},B_{s2})$, where the parameter matrices $A_{s1}$, $A_{s2}$ and $B_{s2}$ are assumed to be unknown. Instead, we are given the input data $u(t_0),u(t_1),\dots, u(t_{T-1})$, the state data $x(t_0),x(t_1),\dots, x(t_{T-1})$ and the state derivative data $\dot{x}(t_0),\dot{x}(t_1),\dots, \dot{x}(t_{T-1})$, where $T$ is a positive integer and $t_0 < t_1 < \cdots < t_{T-1}$. These data are generated by the true system influenced by unknown noise $w(t_0),w(t_1),\dots, w(t_{T-1})$, such that
\begin{equation}\label{realdataeq}
    \dot{\mathcal{X}} = \begin{bmatrix}
        A_{s1} \\ A_{s2}
    \end{bmatrix}\mathcal{F} + \begin{bmatrix}
        0 \\ B_{s2}
    \end{bmatrix} \mathcal{G} \mathcal{U} + \mathcal{W},
\end{equation}
where
$$
    \begin{aligned}
	\dot{\mathcal{X}} &:= \begin{bmatrix} 
            \dot{x}(t_0) & \dot{x}(t_1) & \cdots & \dot{x}(t_{T-1})
	\end{bmatrix},\\
		\mathcal{F} &:= \begin{bmatrix}
			F(x(t_0)) & F(x(t_1)) & \cdots & F(x(t_{T-1}))
		\end{bmatrix},\\
            \mathcal{G} &:= \begin{bmatrix}
			G(x(t_0)) & G(x(t_1)) & \cdots & 	G(x(t_{T-1}))
		\end{bmatrix},\\
		\mathcal{U} &:= \begin{bmatrix}
				u(t_0) & 0 & \cdots & 0\\
				0 & u(t_1) & \cdots & 0\\
				\vdots & \vdots & \ddots & \vdots\\
				0 & 0 & \cdots & u(t_{T-1})
			\end{bmatrix},\\
		\mathcal{W} &:= \begin{bmatrix}
				w(t_0) & w(t_1) & \cdots & w(t_{T-1})
			\end{bmatrix}.
    \end{aligned}
$$
In addition, we define
$$
	\mathcal{X} := \begin{bmatrix}
		x(t_0) & x(t_1) & \cdots & x(t_{T-1})
	\end{bmatrix}.
$$
Although the noise matrix $\mathcal{W}$ is unknown, we assume that it satisfies the following quadratic inequality:
\begin{equation} \label{noisebound}
	\begin{bmatrix}
		1 \\ \vecrz(\mathcal{W}^{\top})
	\end{bmatrix}^{\top} 
		\Phi \begin{bmatrix}
		1\\ \vecrz(\mathcal{W}^{\top})
	\end{bmatrix} \geq 0,
\end{equation}
where
$$
\Phi := \begin{bmatrix}
	\Phi_{11} & 0\\
	0 & -I
\end{bmatrix}
$$
is a given matrix with $\Phi_{11} \geq 0$. The noise model in \eqref{noisebound} is equivalent to the energy bound
\begin{equation} \label{noiseboundenergy}
    \sum_{j=1}^n \sum_{i=0}^{T-1}w_j^2(t_i) \leq \Phi_{11}.
\end{equation}
Note that if the noise samples at every time instant are bounded in norm, that is,
$$
	\|w(t_i)\|^2 \leq \omega \quad \forall i = 0,1,\dots,T-1,
$$
then \eqref{noiseboundenergy} holds with $\Phi_{11} = \omega T$. Of course, the true system may not be the only one that generates the given data. A system $(A_1,A_2,B_2)$ satisfying
\begin{equation}\label{dataeq}
    \dot{\mathcal{X}} = \begin{bmatrix}
        A_{1} \\ A_{2}
    \end{bmatrix}\mathcal{F} + \begin{bmatrix}
        0 \\ B_{2}
    \end{bmatrix} \mathcal{G} \mathcal{U} + \mathcal{W},
\end{equation}
for some $\mathcal{W}$ as in \eqref{noisebound} is called \emph{compatible} with the data. We denote the set of all systems compatible with the data $(\dot{\mathcal{X}},\mathcal{X}, \mathcal{U})$ by $\Sigma$, i.e.,
$$
    \Sigma := \{(A_1,A_2,B_2)\mid \eqref{dataeq} \text{ holds for some } \mathcal{W}\text{ satisfying \eqref{noisebound}} \}.
$$
It follows from \eqref{realdataeq} that $(A_{s1},A_{s2},B_{s2}) \in \Sigma$, but in general $\Sigma$ contains other systems.

In this section, we focus on designing a stabilizing controller for the true system. Since on the basis of the data we cannot distinguish between the true system and any other system in $\Sigma$, the goal is to find a single controller that stabilizes the origin of all systems in $\Sigma$. This motivates the following definition of data informativity \cite{DDCInformativity, DDCinformativityoverview} for stabilization of polynomial systems.

\begin{definition} \label{definition 1}
    {\rm The data $(\dot{\mathcal{X}},\mathcal{X},\mathcal{U})$ are called} informative for stabilization {\rm if there exist a continuously differentiable function $V : \mathbb{R}^n \rightarrow \mathbb{R}$, a continuous controller $K: \mathbb{R}^n \rightarrow \mathbb{R}^m$ and constants $c,r>0$ such that $K(0) = 0$, and \eqref{lemma GAS 1}, \eqref{lemma GAS 2} and \eqref{lemma GAS 3} hold for all $(A_1,A_2,B_2) \in \Sigma$, with $f$ as in \eqref{fx}.
    }
\end{definition}

For a controller $K$ satisfying $K(0) = 0$, the origin of the closed-loop system \eqref{closedloopsystem} is an equilibrium point since $Z(0) = 0$ and $F(x) = H(x)Z(x)$. If all conditions of Definition \ref{definition 1} hold, by Lemma \ref{lemma GAS}, the origin is a globally asymptotically stable equilibrium point of all closed-loop systems obtained by interconnecting any system $(A_1,A_2,B_2) \in \Sigma$ with the controller $K(x)$. Next, we aim to find conditions under which the data $(\dot{\mathcal{X}},\mathcal{X},\mathcal{U})$ are informative for stabilization. 

\subsection{A data-driven approach} \label{section A data-driven approach}
We observe from Theorem \ref{theorem model based} that the main challenge is to ensure that \eqref{bridge to ddc} holds for all systems compatible with the data. To address this, we reformulate the system set $\Sigma$. By vectorizing \eqref{dataeq}, we obtain
\begin{equation} \label{vec W}
	\text{vec}(\dot{\mathcal{X}}^{\top}) = \mathcal{D}^{\top} v+ \text{vec}(\mathcal{W}^{\top}),
\end{equation}
where 
\begin{equation} \label{Def D v}
	\mathcal{D} := \begin{bmatrix}
		I_{n_1} \otimes \mathcal{F} &  0 \\
        0 & I_{n_2} \otimes \mathcal{F}\\
        0 & I_{n_2} \otimes \mathcal{G}\mathcal{U}
	\end{bmatrix}\ \text{and}\ v := \begin{bmatrix}
		\text{vec}(A_1^{\top}) \\ \text{vec}(A_2^{\top}) \\ \text{vec}(B_2^{\top})
	\end{bmatrix}.
\end{equation}
With \eqref{noisebound}, the system $(A_1,A_2,B_2) \in \Sigma$ if and only if the corresponding $v$ satisfies
\begin{equation} \label{bound v}
	\begin{bmatrix}
		1 \\ v
	\end{bmatrix}^{\top} N
	\begin{bmatrix}
		1 \\ v
	\end{bmatrix} \geq 0,
\end{equation}
where
\begin{equation} \label{data matrix N}
\begin{aligned}
    N := \begin{bmatrix}
1 & 0 \\
\vecrz(\dot{\mathcal{X}}^{\top}) & 	- \mathcal{D}^{\top}
\end{bmatrix}^{\top}\Phi \begin{bmatrix}
1 & 0 \\
\vecrz(\dot{\mathcal{X}}^{\top}) & 	- \mathcal{D}^{\top}
\end{bmatrix}.
\end{aligned}
\end{equation}
As such, the set $\Sigma$ of systems compatible with the data is characterized by a quadratic inequality in $v$. We denote $\ell = n_1f+n_2(f+g)$ and then $N \in \mathbb{S}^{1+\ell}$. For simplicity, we define the set
$$
    \mathcal{Z}_{\ell}(N) := \left\{z \in \mathbb{R}^{\ell} \mid \begin{bmatrix}
		1 \\ z
    \end{bmatrix}^{\top} N \begin{bmatrix}
		1 \\ z
	\end{bmatrix} \geq 0 \right\},
$$
which consists of all $v$ satisfying \eqref{bound v}. 

With this notation in place, we formulate a condition in terms of $v$ that implies that \eqref{bridge to ddc} holds for all $(A_1,A_2,B_2) \in \Sigma$. Let $y$ be a vector in $\mathbb{R}^p$. A sufficient condition is that for all $(A_1,A_2,B_2) \in \Sigma$, we have that
\begin{equation} \label{yMy}
    y^{\top}M(x)y \geq \epsilon_2(x)\epsilon_3^{-1}(x)y^{\top}P(x_1)P(x_1)y \  \text{ for all } x,y.
\end{equation}
We rewrite the left-hand side of \eqref{yMy} as $y^{\top}M(x)y = R^{\top}(x,y)v$, where $v$ is defined in \eqref{Def D v}, and
\begin{equation} \label{Rx}
    \begin{aligned}
        &R(x,y) := \\
        &\begin{bmatrix}
                \left(\! \frac{\partial}{\partial x_1}y^{\top}\! P(x_1) y\! \right)^{\!\top}\!\!
                 \otimes  F(x)\! -\! 2\left(\! \frac{\partial Z}{\partial x_1}(x) \! \right)^{\!\top}\! y  \otimes  H(x) P(x_1) y\\
                -2\left(\frac{\partial Z}{\partial x_2}(x)\right)^{\!\top}\!y  \otimes H(x) P(x_1) y\\
                -2\left(\frac{\partial Z}{\partial x_2}(x)\right)^{\!\top}\!y \otimes G(x)L(x)y
            \end{bmatrix}\!.
    \end{aligned}
\end{equation}
Therefore, if the inequality 
\begin{equation} \label{Rv geq}
    R^{\top}(x,y)v - \epsilon_2(x)\epsilon_3^{-1}(x)y^{\top}P(x_1)P(x_1)y \geq 0
\end{equation}
holds for all $x,y$ and $v \in \mathcal{Z}_{\ell}(N)$ then \eqref{bridge to ddc} holds for all $(A_1,A_2,B_2) \in \Sigma$.

Building on the inequality \eqref{Rv geq}, we will introduce our data-driven control approach. Before we do so, we make the following blanket assumption on the data and the polynomial matrices $F$ and $G$.
\begin{assumption} \label{assumption full row rank}
The matrix $
    \begin{bmatrix}
	\mathcal{F}^{\top} & \mathcal{U}^{\top}\mathcal{G}^{\top}
    \end{bmatrix}$ has full column rank.
\end{assumption}

This implies that the data matrix $\mathcal{D}$ has full row rank. Partition
$$
N = \begin{bmatrix}
	N_{11} & N_{12} \\
	N_{21} & N_{22}
\end{bmatrix},
$$
where $N_{11} \in \mathbb{R}$. Since $\mathcal{D}$ has full row rank, we have $N_{22}<0$. Let $N | N_{22}:= N_{11} - N_{12}N_{22}^{-1}N_{21}$ denote the Schur complement of $N$ with respect to $N_{22}$. Since the set $\mathcal{Z}_{\ell}(N)$ is nonempty, it follows from \cite{DDCQMI} that $N | N_{22} \geq 0$. Then, the following specialized S-lemma \cite[Lem.~7]{paper1density} is applicable.
\begin{lemma} \label{lemma slemma}
    Let $N \in \mathbb{S}^{1+r}$, $N_{22} < 0$, $N|N_{22} \geq 0$, $a \in \mathbb{R}$, and $\lambda \in \mathbb{R}^r$. Then, $\lambda^{\top} z + a \geq 0\ \text{ for all } z \in \mathcal{Z}_r(N)$ if and only if
    \begin{equation} \label{lemma b}
	\begin{bmatrix}
			-\lambda^{\top}N_{22}^{-1}N_{21} + a & (N|N_{22})^{\frac{1}{2}}\lambda^{\top} \\
			(N|N_{22})^{\frac{1}{2}}\lambda & (\lambda^{\top}N_{22}^{-1}N_{21}-a)N_{22}
	\end{bmatrix} \geq 0.
    \end{equation}
\end{lemma}
\vspace{1 em}

In comparison to the classical S-lemma \cite{SurveySlemma}, Lemma \ref{lemma slemma} derives conditions under which a linear inequality is implied by a quadratic one. However, by avoiding introducing an additional decision variable, this lemma may reduce conservatism in applying SOS relaxation \cite{paper1density}. We then state the main result of this section.
\begin{theorem} \label{theorem data driven}
    The data $(\dot{\mathcal{X}},\mathcal{X},\mathcal{U})$ are informative for stabilization if there exist constants $\epsilon_1,c,r > 0$, polynomials $\epsilon_2,\epsilon_3\in\mathbb{R}[x]$, and polynomial matrices $P \in \mathbb{R}^{p \times p}[x_1]$ and $K\in\mathbb{R}^{m \times p}[x]$, such that conditions \ref{theorem model based 1}, \ref{theorem model based 2} and \ref{theorem model based 3} of Theorem~\ref{theorem model based} hold, and  
    \begin{equation} \label{theorem data driven sos}
	\begin{aligned}
	    \begin{bmatrix}
                \!- R^{\top\!}(x,y)N_{22}^{-1\!}N_{21}\!   & (N|N_{22})^{\frac{1}{2}}R^{\top\!}(x,y) & \!\!\epsilon_2(x) y^{\!\top\!\!}P(x_1) \\
			(N|N_{22})^{\frac{1}{2}}R(x,y) & \!\!R^{\top\!}(x,y)N_{22}^{-1\!}N_{21}\!N_{22} & 0\\
			\epsilon_2(x)P(x_1)y & 0 & \frac{1}{2} \epsilon_2(x) \epsilon_3(x) I
		\end{bmatrix}&\\
        \in \sos^{1+\ell+p}[x,y]&,
	\end{aligned}	
    \end{equation}
    where $R$ is defined as in \eqref{Rx}. Then, the origin of the closed-loop system, obtained by interconnecting any system in $\Sigma$ and the controller defined as in \eqref{Kx}, is globally asymptotically stable.
\end{theorem}
\begin{proof}
    To prove the theorem, we apply Lemma \ref{lemma GAS} to an arbitrary system $(A_1,A_2,B_2) \in \Sigma$, with the functions $V$ and $f$ given in \eqref{Vx} and \eqref{fx}, respectively. 

    As mentioned in the proof of Theorem \ref{theorem model based}, \eqref{lemma GAS 1} holds. Let $x \in \mathbb{R}^n$ and $y \in \mathbb{R}^p$. By condition \ref{theorem model based 1} of Theorem \ref{theorem model based}, $\epsilon_2(x) \geq 0$ and $\epsilon_3(x)~>~0$. Using a Schur complement argument on the matrix in \eqref{theorem data driven sos}, we have that the matrix inequality 
    \begin{equation} \label{theorem 2a}
        \begin{bmatrix}
			-\lambda^{\top}N_{22}^{-1}N_{21} + 2a & (N|N_{22})^{\frac{1}{2}}\lambda^{\top} \\
			(N|N_{22})^{\frac{1}{2}}\lambda & \lambda^{\top}N_{22}^{-1}N_{21}N_{22}
	\end{bmatrix} \geq 0
    \end{equation}
    holds with $\lambda =  R(x,y)$ and 
    $$
    a = - \epsilon_2(x) \epsilon_3^{-1}(x) y^{\top}P(x_1)P(x_1)y.
    $$ 
    
    We now claim that \eqref{lemma b} also holds. Note that the proof is trivial for $a = 0$; therefore, since $a \leq 0$, it suffices to consider the case $a < 0$. In this case, it is clear that $-\lambda^{\top} N_{22}^{-1}N_{21}>0$. Using a Schur complement argument on the matrix in \eqref{theorem 2a}, we have
    $$
    \begin{aligned}
        (-\lambda^{\!\top}\!N_{22}^{-1}\!N_{21}\! + 2a)(-\lambda^{\!\top}\!N_{22}^{-1}\!N_{21}) - (N|N_{22})\lambda^{\!\top}\! (-N_{22})^{\!-1}\!\lambda \geq 0.
    \end{aligned}
    $$
    This implies
    $$
        (-\lambda^{\top}N_{22}^{-1}N_{21} + a)^2 - (N|N_{22})\lambda^{\top} (-N_{22})^{-1}\lambda \geq 0.
    $$
    Since $-\lambda^{\top}N_{22}^{-1}N_{21}+a>0$, by a Schur complement argument, the matrix inequality \eqref{lemma b} holds. This proves that \eqref{lemma b} holds.
    
    Then, it follows from Lemma \ref{lemma slemma} that \eqref{Rv geq} holds for all $v \in \mathcal{Z}_{\ell}(N)$. Moreover, since \eqref{Rv geq} holds for all $x,y$, we have that \eqref{bridge to ddc} holds for all $(A_1,A_2,B_2) \in \Sigma$. Together with condition \ref{theorem model based 2} of Theorem \ref{theorem model based}, we conclude that \eqref{lemma GAS 2} and \eqref{lemma GAS 3} hold for all $(A_1,A_2,B_2) \in \Sigma$.
\end{proof}

\begin{remark} \label{remark index set}
    In the definition of $\Sigma$, the entries of $A_{s1}$, $A_{s2}$ and $B_{s2}$ are assumed to be fully unknown. However, in physical systems, certain system parameters are often given. For example, in a mechanical system, one state variable may represent the position of a mass, while another represents its velocity. In this case, the row of the matrix $A_{s1}$ corresponding to the position state variable is fully known. We note that our approach in Theorem \ref{theorem data driven} can be adapted to take into account the prior knowledge of entries of $A_{s1}$, $A_{s2}$ and $B_{s2}$. Following \cite{paper1density}, let $\alpha$ be the set of indices for which the corresponding entries of $v$ are given, and let $\hat{\alpha}:= \{1,2,\dots,\ell\} \setminus \alpha$. We define $\mathcal{D}_{\alpha}$ and $\mathcal{D}_{\hat{\alpha}}$ as the submatrices of $\mathcal{D}$ formed by the rows indexed by $\alpha$ and $\hat{\alpha}$, respectively. Define $\hat{N}$ as
    $$
    \begin{bmatrix}
            1 & 0 \\
        \vecrz(\dot{\mathcal{X}}^{\top}) - \mathcal{D}^{\top}_{\alpha}v_{\alpha} & 	- \mathcal{D}^{\top}_{\hat{\alpha}}
        \end{bmatrix}^{\top} \Phi \begin{bmatrix}
            1 & 0 \\
        \vecrz(\dot{\mathcal{X}}^{\top}) - \mathcal{D}^{\top}_{\alpha}v_{\alpha} & 	- \mathcal{D}^{\top}_{\hat{\alpha}}
        \end{bmatrix}.
    $$
    Let $\gamma$ denote the cardinality of $\hat{\alpha}$, and then $\hat{N}\in \mathbb{S}^{1 + \gamma}$. We denote the set of systems compatible with the data and the prior knowledge by ${\Sigma}_\alpha$. The system $(A_1,A_2,B_2)\in {\Sigma}_\alpha$ if and only if the corresponding vector $v_{\hat{\alpha}} \in \mathcal{Z}_{\gamma}(\hat{N})$. Then, Theorem \ref{theorem data driven} can be adapted by replacing \eqref{theorem data driven sos} with \eqref{theorem data driven sos alpha}. If \eqref{theorem data driven sos alpha} and the other conditions of Theorem \ref{theorem data driven} hold, the obtained controller renders the origin of all systems in ${\Sigma}_\alpha$ globally asymptotically stable. As we will illustrate in Example \ref{ex 4 flexjoint}, by incorporating the prior knowledge, the adapted approach is applicable to the model of a nonlinear mechanical system.
    \begin{figure*}[!t]
    \normalsize
    \begin{equation} \label{theorem data driven sos alpha}
		\begin{bmatrix}
                - R^{\top}_{\hat{\alpha}}(x,y)\hat{N}_{22}^{-1}\hat{N}_{21} + R^{\top}_{\alpha}(x,y)v_{\alpha}  & (\hat{N}|\hat{N}_{22})^{\frac{1}{2}}R^{\top}_{\hat{\alpha}}(x,y) & \epsilon_2(x) y^{\top}P(x_1) \\
			(\hat{N}|\hat{N}_{22})^{\frac{1}{2}}R_{\hat{\alpha}}(x,y) & \left(R^{\top}_{\hat{\alpha}}(x,y)\hat{N}_{22}^{-1}\hat{N}_{21} - R^{\top}_{\alpha}(x,y)v_{\alpha}\right)\hat{N}_{22} & 0\\
			\epsilon_2(x)P(x_1)y & 0 & \frac{1}{2}\epsilon_2(x)\epsilon_3(x) I
		\end{bmatrix} \in \sos^{1+\gamma+p}[x,y].
    \end{equation}
    \hrulefill 
    \end{figure*}
\end{remark}

\section{Illustrative examples} \label{section simulation}
In this section, we first illustrate Theorem \ref{theorem data driven} with two numerical examples, followed by a verification of the adapted approach using \eqref{theorem data driven sos alpha} on a model of a mechanical system. The simulations are conducted in MATLAB, using YALMIP \cite{Yalmip} with the SOS module \cite{YalmipSOS}, and the solver MOSEK \cite{mosek}. 
\begin{example}\label{ex 1 datadriven}
    We again consider the system \eqref{ex sys 1}. After taking 
    $$
        F(x) = \begin{bmatrix}
            x_1 & x_2 & x_1^2
        \end{bmatrix}^{\top} \text{ and } G(x) = 1,
    $$
    we have $A_{s1} = \begin{bmatrix}
        0 & 1 & -1 
        \end{bmatrix}$, $A_{s2} = \begin{bmatrix}
        0 & 0 & 0
    \end{bmatrix}$ and $B_{s2} = 1$. Let $x(0) = \begin{bmatrix}
		-1 & 1
    \end{bmatrix}^{\top}$ and $u(t)= 2\sin(5t)+\cos(3t)$ for $t \in [0,5]$. We collect $T=4$ data samples from the system at the time instants $t_0 = 0.60$, $t_1 = 0.67$, $t_2 = 0.74$ and $t_3 = 0.81$. During the experiment, the noise samples are drawn independently and uniformly at random from the ball $\{r \in \mathbb{R}^2 \mid \|r\| \leq \omega \}$, where $\omega = 0.01$. Define $\Phi_{11} = \omega^2 T $. The data matrices are
    $$
    	\begin{aligned}
    		\dot{\mathcal{X}} &= \begin{bmatrix}
    			1.8948  &  1.9749  &  1.9618  &  1.8536\\
                    0.0478  & -0.8341  & -1.6731  & -2.3295
    		\end{bmatrix},\\
    		\mathcal{X} &= \begin{bmatrix}
    			-0.4769 &  -0.3409 &  -0.2025  & -0.0682\\
                    2.1206  &  2.0930  &  2.0048  &  1.8637
    		\end{bmatrix},\\
    		\mathcal{U} &= \operatorname{diag}(0.0550, -0.8390, -1.6642, -2.3344).
            \end{aligned}
    $$
    Furthermore, we have
    $$
    	\begin{aligned}
    		\mathcal{F} &= \begin{bmatrix}
    			-0.4769 &  -0.3409 &  -0.2025  & -0.0682\\
                    2.1206  &  2.0930  &  2.0048  &  1.8637\\
    			0.2275  &  0.1162  &  0.0410  &  0.0046
    		\end{bmatrix},\\
    		\mathcal{GU} &= \begin{bmatrix}
    			0.0550  & -0.8390  & -1.6642 &  -2.3344
    		\end{bmatrix}.
    	\end{aligned}
    $$
    Choose $Z(x) = x$ and
    \begin{equation} \label{Hx}
        H(x) = \begin{bmatrix}
        1 & 0\\
        0 & 1\\
        x_1 & 0
    \end{bmatrix}.
    \end{equation}
    After choosing $\epsilon_1 = 0.1$, $\epsilon_2 = 0.01$ and $\epsilon_3(x) = 2 + 2x_1^2$, conditions \ref{theorem model based 1} and \ref{theorem model based 2} of Theorem \ref{theorem model based} hold. We set the maximum degrees of the entries of $P\in \mathbb{R}^{2 \times 2}[x_1]$ and $L\in \mathbb{R}^{1 \times 2}[x]$ to 4 and 6, respectively. Define $R$ as in \eqref{Rx}. We formulate an SOS programming imposing condition \ref{theorem model based 3} of Theorem \ref{theorem model based} and the SOS constraint \eqref{theorem data driven sos}, and we solve this for $P$ and $L$. We obtain 
    $$
    \begin{aligned}
        P(x_1) &= \begin{bmatrix}
            0.3791 & -0.4709-0.2067x_1^2\\
            -0.4709-0.2067x_1^2 & 2.9374 + 0.2580 x_1^4
        \end{bmatrix},\\
        L(x) &= \begin{bmatrix}
            -2.3092 & - 1.6748 - 1.2095 x_2^2 - 1.5350 x_1^6
        \end{bmatrix}.
    \end{aligned}
    $$
    With $P(x_1)$ and $L(x)$ in place, we calculate $V(x)$ and $K(x)$, as defined in \eqref{ex1 Vx} and \eqref{ex1 Kx}, respectively, where
    $$
    \begin{aligned}
        &\begin{aligned}
            \eta(x) = &\ 2.9374 x_1^2+0.9418 x_1x_2+0.3791 x_2^2+0.4134 x_1^3x_2\\
            &+0.2580 x_1^6,
        \end{aligned}\\
        &\begin{aligned}
            \xi(x) = & -7.5717 x_1-1.7223 x_2-0.3462 x_1^3-0.4773 x_1^2x_2\\
            &-0.5696 x_1x_2^2 -0.4585 x_2^3 -0.5958 x_1^5-0.2500 x_1^3x_2^2\\
            &-0.7228 x_1^7-0.5819 x_1^6x_2-0.3173 x_1^9,
        \end{aligned}\\
        & \det{P(x_1)} = 0.8918 - 0.1947 x_1^2 + 0.0551 x_1^4.
    \end{aligned}
    $$
    This verifies global asymptotic stability of the closed-loop system with respect to the system \eqref{ex sys 1} and the obtained controller $K$. The phase portrait of the closed-loop system and the level sets of $V(x)$ are illustrated in Fig.~\ref{fig 2}.
    \begin{figure}[!t]
        \centering
        \subfloat[]{
    	\includegraphics[clip,trim=45 30 40 60,width=0.85\columnwidth]{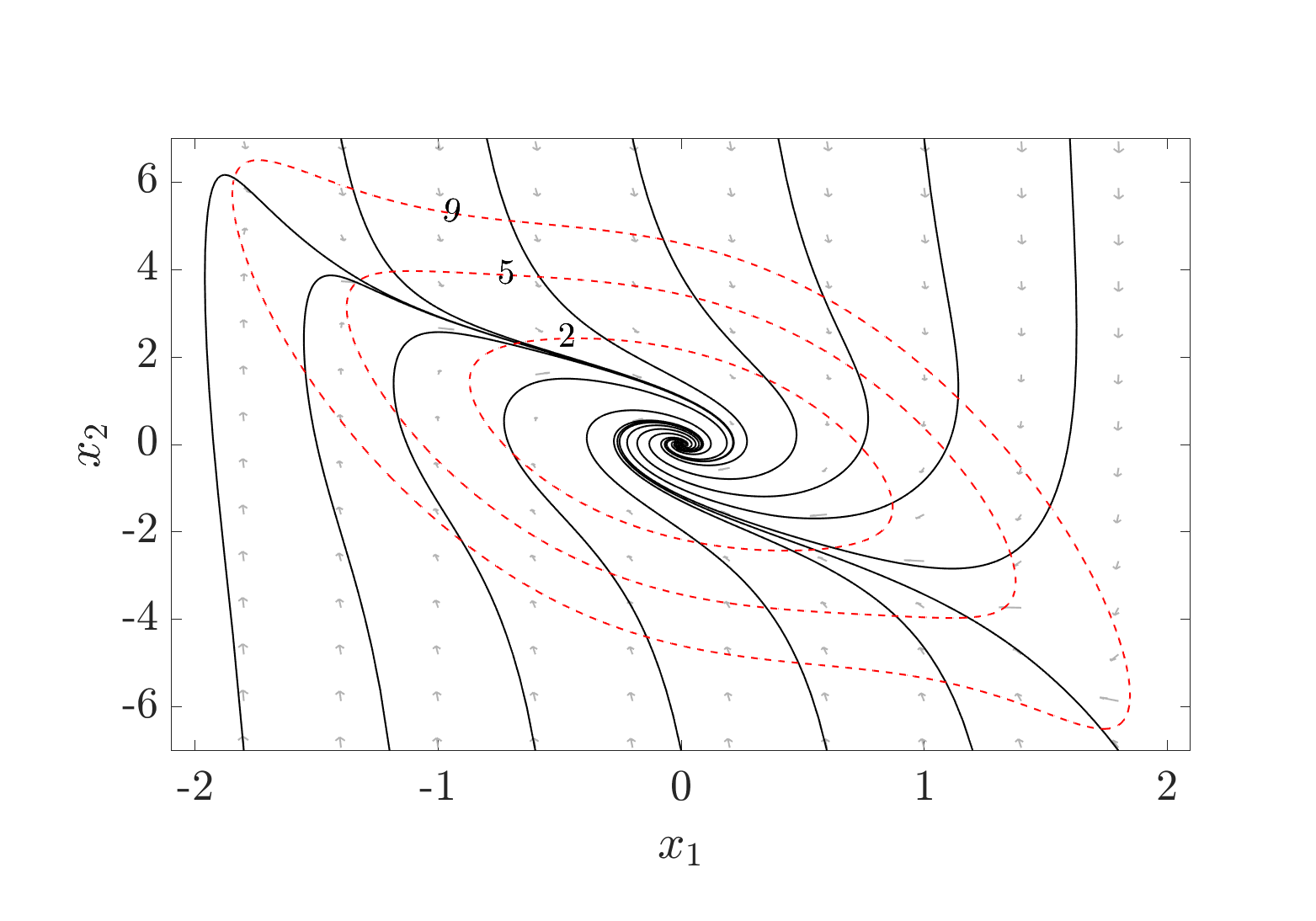}
    	\label{fig 2}
        }
        \hfill
        \subfloat[]{
		\includegraphics[clip,trim=45 30 40 60,width=0.85\columnwidth]{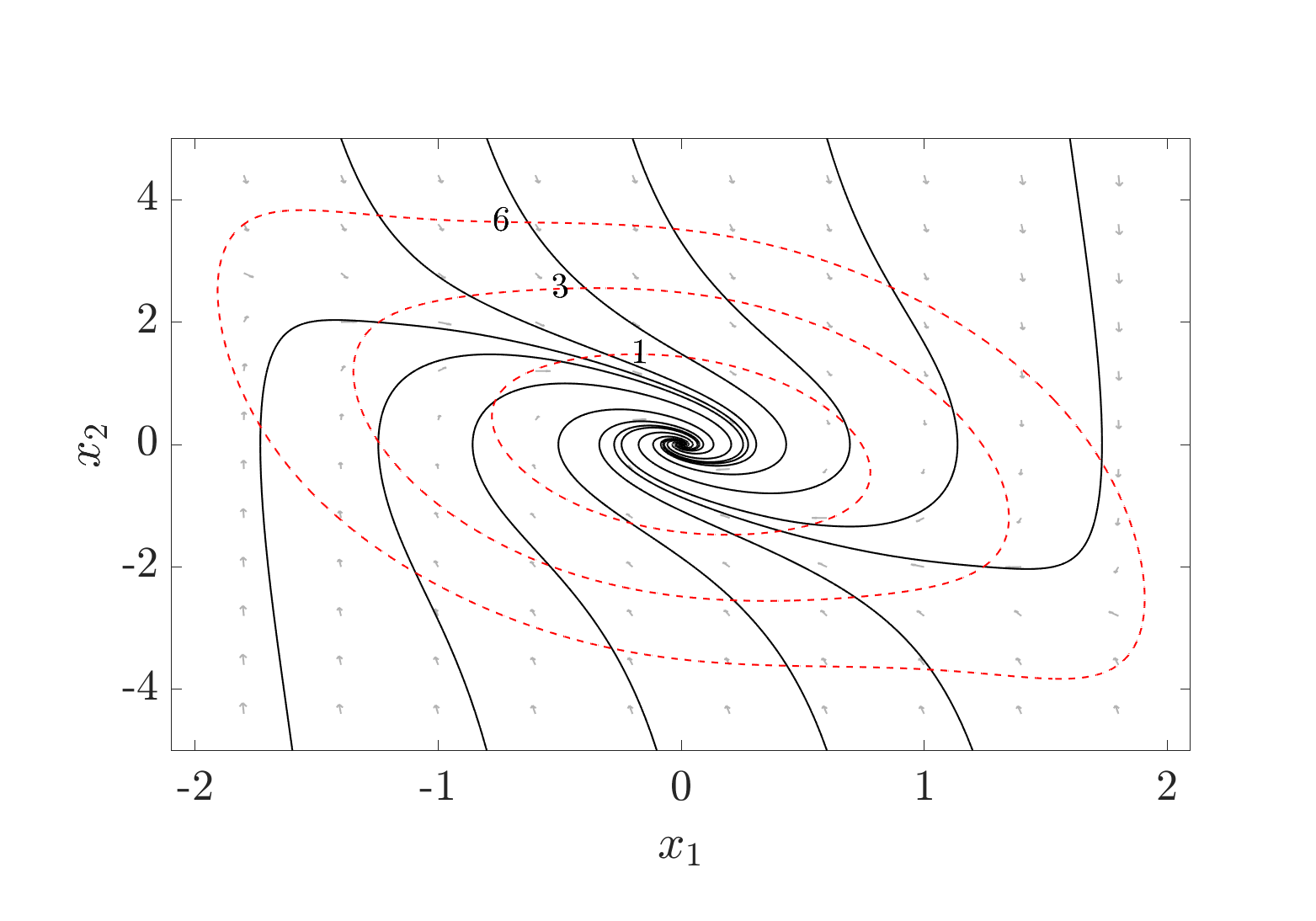}
		\label{fig 3}
        }
        \caption{Phase portrait of the closed-loop system and the level sets of $V(x)$ in Example \ref{ex 1 datadriven} (see \protect\subref{fig 2}) and Example \ref{ex 3} (see \protect\subref{fig 3}). The gray arrows depict the closed-loop vector field and the black lines are trajectories starting from the edges and converging to the origin. The level sets are indicated by red dashed lines.}
    \end{figure}
\end{example}

\begin{example} \label{ex 3}
    Consider the following system
    \begin{equation} \label{sys ex3}
    \begin{aligned}
        \dot{x}_1 = x_2,\quad
        \dot{x}_2 = -x_1+x_2 -x_1^3 +u.
    \end{aligned}
    \end{equation}
    Let
    $$
    F(x) = \begin{bmatrix}
        x_1 & x_2 & x_1^3
    \end{bmatrix}^{\top} \text{ and } G(x) = 1.
    $$
    Then, $A_{s1} = \begin{bmatrix}
        0 & 1 & 0
        \end{bmatrix}$, $A_{s2} = \begin{bmatrix}
        -1 & 1 & -1
    \end{bmatrix}$ and $B_{s2} = 1$. We choose $Z(x) = x$, and all $H(x)$ are of the form 
    $$
    H(x) = \begin{bmatrix}
        1 & 0\\
        0 & 1\\
        x_1^2 & 0
    \end{bmatrix} + J(x) \begin{bmatrix}
        x_2 & -x_1
    \end{bmatrix},
    $$
    where $J \in \mathbb{R}^3[x]$ is arbitrary. It can be verified that for the true system \eqref{sys ex3}, there exist $L \in \mathbb{R}^{1 \times 2}[x]$ and $P \in \mathbb{S}^2$ such that the matrix inequality $M(x) > 0$ holds for all $x \neq 0$. However, in the context of data-driven control, restricting $P$ to be constant is too conservative to find a stabilizer, since the inequality must hold for all systems compatible with the data. To illustrate the conservatism, suppose that the set $\Sigma$ has nonempty interior. Let $a$ be a sufficiently small positive constant such that $(\bar{A}_1,A_{s2},B_{s2}) \in \Sigma$, where $\bar{A}_1 := \begin{bmatrix}
        0 & 1 & a
    \end{bmatrix}$. For $(\bar{A}_1,A_{s2},B_{s2})$, suppose that there exist $L \in \mathbb{R}^{1 \times 2}[x]$ and $P \in \mathbb{S}^2$ such that $M(x) > 0$ for all $x \neq 0$. The $(1,1)$-entry of $M(x)$ is
    $$
        M_{11}(x) \! = \! -2(J_2(x)\!+\!aJ_3(x))(P_{11}x_2\!-\!P_{12}x_1)\! -\! 2(P_{11}ax_1^2\!+\!P_{21}).
    $$ 
    We now choose $x_2 = P_{12}x_1/P_{11}$, which leads to $M_{11}(x) = -2(P_{11}ax_1^2+P_{21})$. Therefore, there exists a sufficiently large $x_1$ such that $M_{11}(x)<0$, which is a contradiction. We conclude that if $P$ is constant, $M(x)$ cannot be positive definite for all nonzero $x$ and all systems in the set $\Sigma$.
    
    Next, we apply Theorem \ref{theorem data driven} for $L\in \mathbb{R}^{1 \times 2}[x]$ and $P\in \mathbb{R}^{2 \times 2}[x_1]$. We define $H(x)$ as in \eqref{Hx}. Let $x(0) = \begin{bmatrix}
		-0.5 & 0.5
    \end{bmatrix}^{\top}$ and $u(t)= 3\sin(2t)-5\cos(t)$ for $t \in [0,10]$. We collect $T=4$ data samples from the system at the time instants $t_0 = 3.10$, $t_1 =~3.18$, $t_2 = 3.26$ and $t_3 = 3.34$. During the experiment, the noise samples are drawn independently and uniformly at random from the ball $\{r \in \mathbb{R}^2 \mid \|r\| \leq \omega \}$, where $\omega = 0.02$. Define $\Phi_{11} = \omega^2 T $. The data matrices
    $$
    	\begin{aligned}
    		\dot{\mathcal{X}} &= \begin{bmatrix}
    			-3.6951 &  -4.6837 &  -5.3252  & -5.6333\\
                    -14.9236 & -10.1249 &  -5.7669  & -2.5443
    		\end{bmatrix},\\
    		\mathcal{X} &= \begin{bmatrix}
    			2.3862  &  2.0483  &  1.6453  &  1.2050\\
                    -3.6915 &  -4.6933 &  -5.3216  & -5.6447\\
    		\end{bmatrix},\\
    		\mathcal{U} &= \operatorname{diag}(4.7464, 5.2265, 5.6688, 6.0614).
            \end{aligned}
    $$
    Furthermore, we have
    $$
    \begin{aligned}
        \mathcal{F} &= \begin{bmatrix}
    			2.3862  &  2.0483  &  1.6453  &  1.2050\\
                    -3.6915 &  -4.6933 &  -5.3216  & -5.6447\\
                    13.5877  &  8.5933 &   4.4542  &  1.7496
    		\end{bmatrix},\\
        \mathcal{GU} &= \begin{bmatrix}
    			4.7464  &  5.2265  &  5.6688  &  6.0614
    		\end{bmatrix}.
    \end{aligned}		
    $$
    After choosing $\epsilon_1 = 0.1$, $\epsilon_2 = 0.01$ and $\epsilon_3(x) = 2 + 2x_1^2$, conditions \ref{theorem model based 1} and \ref{theorem model based 2} of Theorem \ref{theorem model based} hold. We set the maximum degrees of the entries of $P$ and $L$ to 4 and 6, respectively. Define $R$ as in \eqref{Rx}. We formulate an SOS programming imposing condition \ref{theorem model based 3} of Theorem \ref{theorem model based} and the SOS constraint \eqref{theorem data driven sos}. By solving this for $P$ and $L$, we obtain 
    $$
    \begin{aligned}
        P(x_1) =& \begin{bmatrix}
            0.6067 & -0.2700-0.1455x_1^2\\
            -0.2700-0.1455x_1^2 & 2.1736 + 0.1614x_1^4
        \end{bmatrix},\\
        L(x) =& \begin{bmatrix}
        -0.9228 & -4.0530-0.8888x_2^2 -1.5962x_1^6
    \end{bmatrix}.
    \end{aligned}
    $$
    With $P(x_1)$ and $L(x)$ in place, we calculate $V(x)$ and $K(x)$, as defined in \eqref{ex1 Vx} and \eqref{ex1 Kx}, respectively, where
    $$
    \begin{aligned}
        &\begin{aligned}
            \eta(x) = &\ 2.1736 x_1^2+0.5400x_1x_2+0.6067x_2^2+0.2910x_1^3x_2 \\
             &+ 0.1614x_1^6,
        \end{aligned}\\
        &\begin{aligned}
            \xi(x) = & -3.1001 x_1 - 2.7081x_2 - 0.5897x_1^3 - 0.1343x_1^2x_2\\
            &-0.2400 x_1x_2^2 - 0.5392x_2^3 - 0.1489x_1^5 - 0.1293x_1^3x_2^2 \\
            &- 0.4310x_1^7 - 0.9684x_1^6x_2 - 0.2322x_1^9,
        \end{aligned}\\
        &\det{P(x_1)} = 1.2458-0.0788x_1^2+0.0768 x_1^4.
    \end{aligned}
    $$
    This verifies global asymptotic stability of the closed-loop system with respect to the system \eqref{sys ex3} and the obtained controller $K$. The phase portrait of the closed-loop system and the level sets of $V(x)$ are illustrated in Fig.~\ref{fig 3}.
\end{example}
\begin{example} \label{ex 4 flexjoint}
    Inspired by the so-called hardening spring \cite[Sec~1.2.3]{khalil2002nonlinear}, we consider the following model of a mass-spring system, as depicted in Fig.~\ref{figmassspring},
    $$
    \begin{aligned}
        m_1\ddot{q}(t)\! +\! c_1\dot{q}(t) \!+\! k_1(q(t) \!- \!\theta(t))\! + \!k_2(q(t) \!- \!\theta(t))^3\! &=\! 0,\\
        m_2\ddot{\theta}(t)\! +\! c_2 \dot{\theta}(t)\! +\!  k_1(\theta(t)\! -\! q(t))\! + \!k_2(\theta(t)\! -\! q(t))^3\! &=\! \psi(t),
    \end{aligned}
    $$
    where $q(t),\theta(t),\dot{q}(t), \dot{\theta}(t)\in \mathbb{R}$. We denote the horizontal displacement of mass 1 and 2 by $q(t)$ and $\theta(t)$, respectively. The force applied to mass 2 is given by $\psi(t)$. Moreover, $m_1$ and $c_1$ represent the mass and friction coefficient of mass 1, while $m_2$ and $c_2$ denote those of mass 2. The coefficients of the hardening spring are denoted by $k_1$ and $k_2$. In this example, we let the true parameters take $m_1 = 2\ \mathrm{kg}$, $m_2 = 1\ \mathrm{kg}$, $k_1 = 1\ \mathrm{N / m}$, $k_2 = 1\ \mathrm{N / m^3}$ and $c_1 = c_2 = 0.1\ \mathrm{N \cdot s / m}$. The units of $q$, $\dot{q}$, $\theta$, $\dot{\theta}$, $\psi$ and $t$ are $\mathrm{m}$, $\mathrm{m}$, $\mathrm{m/s}$, $\mathrm{m/s}$, $\mathrm{N}$ and $\mathrm{s}$, respectively. We consider $\psi$ as the input and $q$, $q-\theta$, $\dot{q}$ and $\dot{\theta}$ as the state variables. We aim to design a controller to stabilize the equilibrium point $q = 0$, $q-\theta = 0$, $\dot{q} = 0$ and $\dot{\theta} = 0$. After defining $\zeta := \begin{bmatrix}
		q & q -\theta & \dot{q} &\dot{\theta}
	\end{bmatrix}^{\top}$ and $u := \psi$, the state-space equations are
    $$
    \begin{aligned}
        \dot{\zeta}_1 &= \zeta_3,\\
        \dot{\zeta}_2 &= \zeta_3 - \zeta_4,\\
        \dot{\zeta}_3 &= - 0.05\zeta_3 - 0.5\zeta_2 - 0.5 \zeta_2^3,\\
        \dot{\zeta}_4 &= - 0.1\zeta_4 + \zeta_2 +  \zeta_2^3 +  u.
    \end{aligned}
    $$
    After taking 
    $$
        F(\zeta) = \begin{bmatrix}
		\zeta_2 & \zeta_3 & \zeta_4 & \zeta_2^3
	\end{bmatrix}^{\top} \  \text{and} \  G(\zeta) = 1,
    $$
    we have
    $$
    \begin{aligned}
        A_{s1} &= \begin{bmatrix}
		0 & 1 & 0 & 0\\
		0 & 1 & -1 & 0\\
            -0.5 & -0.05 & 0 & -0.5\end{bmatrix},\\
    A_{s2} &= \begin{bmatrix}
		1 & 0 & -0.1 & 1
	\end{bmatrix},
    \end{aligned}
    $$
    and $B_{s2} = 1$. We assume that all zero entries of $A_{s1}$ and $A_{s2}$ are known. Considering the physical meanings of the state variables, it is also natural to assume that the $(1,2)$, $(2,2)$ and $(2,3)$-entries of $A_{s1}$ are known, while the remaining nonzero entries are unknown. As such, in this example, the index set $\hat{\alpha} = \{9,10,12,13,15,16,17\}$. 
    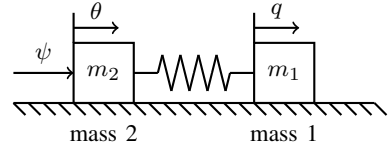
\begin{figure}[!t]
        \centering
        \begin{tikzpicture}[scale=0.8]
            \draw[thick] (-1,0) rectangle (-2,1); 
            \node at (-1.5,-0.5) {mass 2};
            \node at (-1.5,0.5) {$m_2$};
            \draw[thick] (1,0) rectangle (2,1);
            \node at (1.5,-0.5) {mass 1};
            \node at (1.5,0.5) {$m_1$};
            \draw[thick] (-3,0) -- (3,0);
            \foreach \x in {-3, -2.75, ..., 3} {
                \draw[thick] (\x,0) -- (\x+0.2,-0.2);
            }
            \draw[thick] (-1,0.5) -- (-0.6,0.5);
            \foreach \x in {-0.6, -0.3, 0, 0.3} {
                \draw[thick] (\x,0.5) -- (\x+0.075,0.8);
                \draw[thick] (\x+0.075,0.8) -- (\x+0.225,0.2);
                \draw[thick] (\x+0.225,0.2) -- (\x+0.3,0.5);
            }
            \draw[thick] (0.6,0.5) -- (1,0.5);
    
            \draw[thick] (-2,1) -- (-2,1.5);
            \draw[->, thick, black]  (-2,1.25) -- (-1.25,1.25) node[midway, above] {$\theta$};
            \draw[thick] (1,1) -- (1,1.5);
            \draw[->, thick, black]  (1,1.25) -- (1.75,1.25) node[midway, above] {$q$};
            \draw[->, thick, black]  (-3,0.5) -- (-2,0.5) node[midway, above] {$\psi$};
        \end{tikzpicture}
        \caption{The mass-spring system with a hardening spring.}
        \label{figmassspring}
    \end{figure}
    
    Let $\zeta(0) = \begin{bmatrix}
		-1 & 1 & 0 & 0
    \end{bmatrix}^{\top}$ and $u(t)= 3\sin(2t)+2\cos(t)$ for $t \in [0,5]$. We collect $T = 5$ data samples from the system at the time instants $t_0 = 2$, $t_1 = 2.07$, $t_2 = 2.14$, $t_3 = 2.21$ and $t_4 = 2.28$. During the experiment, the noise samples are drawn independently and uniformly at random from the ball $\{r \in \mathbb{R}^4 \mid \|r\| \leq \omega \}$, where $\omega = 0.01$. Define $\Phi_{11} = \omega^2 T $. The data matrices are
    $$
	\begin{aligned}
        \dot{\mathcal{X}}\! &=\! \begin{bmatrix}
		      3.1283  &\!  3.0981 &\!   3.0274  &\!  2.8962  &\!  2.6564\\
                5.5033  &\!  5.6383  &\!  5.6946  &\!  5.5679  &\!  5.1572\\
                -0.3174 &\!  -0.6682 &\!  -1.3468 &\!  -2.5376 &\!  -4.3088\\
                -2.5351 &\!  -2.1930 &\!  -1.1381 &\!   0.9877 &\!   4.3452
	\end{bmatrix},\\
        \mathcal{X}\! &=\! \begin{bmatrix}
                0.7683  &\!  0.9862 &\!   1.2006 &\!   1.4083 &\!   1.6032\\
                0.3046  &\!  0.6949 &\!   1.0923 &\!   1.4880 &\!   1.8657\\
                3.1269  &\!  3.0936 &\!   3.0254 &\!   2.8928  &\!  2.6564\\
               -2.3780  &\! -2.5465 &\!  -2.6684 &\!  -2.6809 &\!  -2.5014\\
		\end{bmatrix},\\
        \mathcal{U} &= \operatorname{diag}(-3.1027, -3.4793, -3.8018, -4.0658, -4.2677).
    \end{aligned}
    $$
    Moreover, we have
    $$
	\begin{aligned}
	&\mathcal{F}\! =\! \begin{bmatrix}
		0.3046  &\!  0.6949 &\!   1.0923 &\!   1.4880 &\!   1.8657\\
            3.1269  &\!  3.0936 &\!   3.0254 &\!   2.8928  &\!  2.6564\\
            -2.3780  &\! -2.5465 &\!  -2.6684 &\!  -2.6809 &\!  -2.5014\\
            0.0283  &\!  0.3356  &\!  1.3033   &\! 3.2949  &\!  6.4945
	\end{bmatrix},\\
	&\mathcal{GU} = \begin{bmatrix}
		-3.1027  &\! -3.4793  &\! -3.8018 &\!  -4.0658 &\!  -4.2677
	\end{bmatrix}.
	\end{aligned}
    $$
    Then, we choose $Z(\zeta) = \zeta$ and
    $$
    H(\zeta) = \begin{bmatrix}
        0 & 1 & 0 & 0\\
        0 & 0 & 1 & 0\\
        0 & 0 & 0 & 1\\
        0 & \zeta_2^2 & 0 & 0
    \end{bmatrix}.
    $$
    After choosing $\epsilon_1 = 0.1$, $\epsilon_2 = 0.01$ and $\epsilon_3(x) = 2 + 2\zeta_1^2 + 2\zeta_2^2 + 2\zeta_3^2$, conditions \ref{theorem model based 1} and \ref{theorem model based 2} of Theorem \ref{theorem model based} hold. We set the maximum degree of the entries of $P\in \mathbb{R}^{4 \times 4}[\zeta_2]$ and $L\in \mathbb{R}^{1 \times 4}[\zeta]$ to 4 and 6, respectively. Define $R$ as in \eqref{Rx}. We formulate an SOS programming imposing condition \ref{theorem model based 3} of Theorem \ref{theorem model based} and the SOS constraint \eqref{theorem data driven sos alpha}, and we solve this for $P$ and $L$. We obtain 
    \begin{equation*}
    \begin{aligned}
    P(\zeta_2) &= \begin{bmatrix}
        6.8942 & 0 & -0.9845 & -0.7844\\
        0 & 0.4613 & 0.2756 & P_{24}(\zeta_2)\\
        -0.9845 & 0.2756 & 0.6371 & P_{34}(\zeta_2)\\
        -0.7844 & P_{42}(\zeta_2) & P_{34}(\zeta_2) & P_{44}(\zeta_2)
    \end{bmatrix},\\
    L(\zeta) &= \begin{bmatrix}
        -0.1888 & 0.1709 & -0.0788 & L_4(\zeta)
    \end{bmatrix},
    \end{aligned}
    \end{equation*}
    where 
    $$
    \begin{aligned}
        P_{24}(\zeta_2) &= P_{42}(\zeta_2) = 0.4522 + 0.1062\zeta_2^2,\\
        P_{34}(\zeta_2) &= P_{43}(\zeta_2) = 0.4036 - 0.0540\zeta_2^2,\\
        P_{44}(\zeta_2) &= 1.3096 + 0.2286\zeta_2^4,\\
        L_4(\zeta) &= -3.2414 - 1.1473 \zeta_2^2\zeta_3^2 - 1.1505 \zeta_2^2\zeta_4^2 - 1.7049\zeta_2^6.
    \end{aligned}
    $$
    With $P(\zeta_2)$ and $L(\zeta)$ in place, we can calculate the Lyapunov function candidate $V(\zeta) = \zeta^{\top} P^{-1}(\zeta_2)\zeta$ and the controller $K(\zeta) = L(\zeta)P^{-1}(\zeta_2)\zeta$. In Fig.~\ref{fig 4}, we present 4 sets of trajectories of the true closed-loop system. Each entry of the initial state $\zeta(0)$ is chosen independently and uniformly at random from the interval $[-2, 2]$. These figures further illustrate that the origin of the true closed-loop system is asymptotically stable.
    \begin{figure}[!t]
	\begin{center}
        \includegraphics[clip,trim=55 10 60 35,width=0.90\columnwidth]{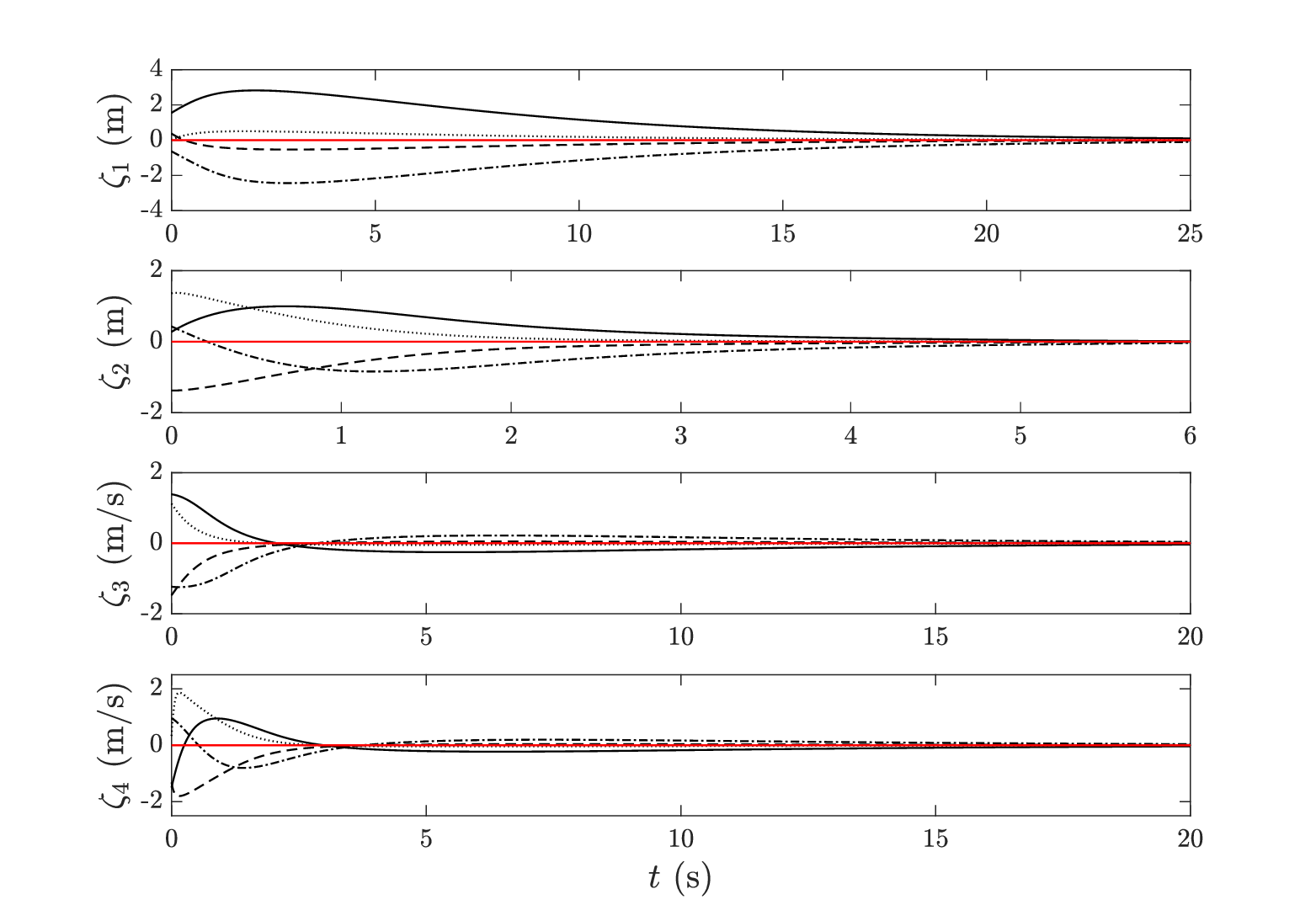} 
        \caption{Trajectories of the true closed-loop system for four different initial states. The zero state in each subplot is depicted by a red line.}
        \label{fig 4} 
        \end{center}
    \end{figure}
\end{example}

\section{Conclusion} \label{section conclusion}
In this study, we have considered a class of polynomial systems in which some state derivatives are not directly affected by the inputs. We have proposed novel stabilization methods in both model-based and data-driven settings. In comparison to the existing approach in \cite[Thm.~6]{Originalmiddle}, our model-based method can be applied to Lyapunov functions that are not radially unbounded, while still guaranteeing global asymptotic stability of the closed-loop system. This result is then extended to the data-driven setting through a specialized S-lemma \cite{paper1density}. In addition to handling noise, our data-driven approach can be adapted to incorporate prior knowledge of the system parameters. 

For future research, it would be of interest to explore the potential benefits of a broader class of Lyapunov function candidates and controllers, in addition to improved feasibility. One possible direction is to include performance specifications, such as constraints on input energy. Finally, another open problem is applying our methods to real-world systems, such as a flexible-joint robot.

\section*{References}
\vspace{-2 em}
\bibliographystyle{IEEEtran}
\bibliography{ReferencePapers}

\end{document}